\documentclass[12pt]{amsart}
\usepackage{amsmath,ifthen,amsthm,srcltx,amsopn,amssymb,amsfonts}
\usepackage{amscd}
\usepackage[english]{babel}

\usepackage[T1]{fontenc}
\usepackage[utf8x]{inputenc}
\usepackage{hyperref}
\usepackage{psfrag}
\usepackage{delarray,graphicx}
\usepackage{comment}
\usepackage[all,cmtip]{xy}

\usepackage{graphicx}

\usepackage{amssymb,latexsym}
\usepackage{color}

\usepackage{eucal}

\title[Discreteness of peripheral representations]{Deformation of hyperbolic
manifolds in $\PGL(n,\C)$ and discreteness of the peripheral representations}

\author{Antonin Guilloux} 
\address{Institut de Math\'ematiques de Jussieu \\
Unit\'e Mixte de Recherche 7586 du CNRS \\
Universit\'e Pierre et Marie Curie \\
4, place Jussieu 75252 Paris Cedex 05, France \\}
\email{aguillou@math.jussieu.fr}
\urladdr{http://people.math.jussieu.fr/~aguillou}

\begin{document}

\renewcommand{\H}{\mathbf{H}}
\newcommand{\Z}{\mathbf{Z}}
\newcommand{\R}{\mathbf{R}}
\newcommand{\C}{\mathbf{C}}
\newcommand{\PGL}{\mathrm{PGL}}
\newcommand{\SL}{\mathrm{SL}}
\newcommand{\Hom}{\mathrm{Hom}}
\newcommand{\periph}{\mathrm{periph}}
\newcommand{\geom}{\mathrm{geom}}
\newcommand{\Hol}{\mathrm{Hol}}

\newtheorem{proposition}{Proposition}
\newtheorem{definition}{Definition}
\newtheorem{theorem}{Theorem}
\newtheorem{lemma}{Lemma}
\newtheorem{fact}{Fact}
\newtheorem{remark}{Remark}

\begin{abstract}  
Let $M$ be a cusped hyperbolic $3$-manifold, e.g. a knot complement. Thurston
\cite{thurston-gt3m} showed that the space of deformations of its fundamental
group in $\PGL(2,\C)$ (up to conjugation) is of complex dimension the number
$\nu$ of cusps near the hyperbolic representation. It
seems natural to ask whether some representations remain discrete after
deformation. The answer is generically not. A simple reason for it lies inside
the cusps: the degeneracy of the peripheral representation (i.e.
representations of fundamental groups of the $\nu$ peripheral tori). They indeed
generically become non discrete, except for a countable
set. This last set corresponds to hyperbolic Dehn surgeries on $M$, for which
the peripheral representation is no more faithful.

We work here in the framework of $\PGL(n,\C)$. The hyperbolic structure lifts,
via the $n$-dimensional irreducible representation, to a representation
$\rho_{\geom}$. We know from the work of Menal-Ferrer and Porti
\cite{MenalFerrer-Porti} that the space of deformations of
$\rho_\textrm{geom}$ has complex dimension
$(n-1)\nu$. 

We prove here that, unlike the $\PGL(2)$-case, the generic behaviour becomes the
discreteness (and faithfulness) of the peripheral representation: in a
neighbourhood of the geometric representation, the non-discrete peripheral
representations are contained in a real analytic subvariety of codimension $\geq 1$.
\end{abstract}

\maketitle

\section{Introduction}

Let $M$ be a complete orientable hyperbolic $3$-manifold with $\nu\geq 1$ cusps,
e.g. a
knot complement. For a Lie group $G$, consider the space $\chi(M,G)$ of
representations of its fundamental group modulo conjugacy:
$$\chi(M,G)=\Hom(\pi_1(M),G)/G.$$
Let $T_1,\ldots,T_\nu$ be the peripheral tori of $M$ (so that the cusps are of
the form $T_i\times[0,\infty)$). We choose once for all a longitude $l_i$ and a
meridian $m_i$ for each of them.

In the case $G=\PGL(2,\C)$, a natural point to consider in the space
$\chi(M,\PGL(2,\C))$ is the (class of the) monodromy $[\rho_\textrm{hyp}]$ of
the hyperbolic
structure. Following Thurston \cite{thurston-gt3m}, Neumann-Zagier \cite{neumann-zagier} showed
it is a smooth point and the complex dimension of $\chi$ at this point equals to
the number $\nu$ of cusps (see also \cite{kapovich}). Moreover the (hyperbolic) length of the longitudes
$l_i$ -- or  directly related parameters as the trace of their holonomies
-- are natural local parameters for this space. In other words, in a
neighbourhood of $[\rho_{\textrm{hyp}}]$, the deformations are described by their
restriction on the tori. This restriction, called peripheral
representation, will be a central object in this paper:
\begin{definition}
For any $\rho \in \chi(M,G)$, its \emph{peripheral representation} 
$\rho_{\periph}$ is the collection of the restrictions of $\rho$ to
$\pi_1(T_i)$. 
\end{definition}
The above mentioned phenomenon -- that the peripheral representations prescribe the whole representation -- is called \emph{local rigidity} (around
$[\rho_{\textrm{hyp}}]$).

Still in the $\PGL(2,\C)$-case, we know which representations remain
discrete in a neighbourhood of the hyperbolic one \cite{thurston-gt3m,
neumann-zagier}. Indeed, after deformation of the hyperbolic representation,
the new
representation becomes
generically non discrete. And a simple reason for it lies in the
peripheral representations: they already are not discrete, except for a
countable set. This last set corresponds to hyperbolic Dehn surgeries on $M$
(or a finite covering) for
which the whole representation is indeed discrete. Beware that in this situation
the peripheral representations are no more faithful.

We work in this paper in the framework of $\PGL(n,\C)$. The hyperbolic structure
lifts, via the $n$-dimensional irreducible representation $$r_n\: : \:
\PGL(2,\C)\to \PGL(n,\C),$$ to an irreducible representation
$\rho_{\geom}=r_n\circ \rho_\textrm{hyp}$ called the \emph{geometric
representation}.
Let us mention that, when $n=3$, the irreducible representation $r_3$ is more
widely known as the adjoint representation $\textrm{Ad}$. The problem of local
rigidity around $\rho_\geom$ has already been studied by Menal-Ferrer and
Porti in
\cite{MenalFerrer-Porti} and shown to hold also for
$G=\PGL(n,\C)$, see theorem \ref{thm:mf-p}. We hence know that the space of deformations near $\rho_{\geom}$
has complex dimension $(n-1)\nu$ and that the symmetric functions of the eigenvalues of
$\rho(l_i)$ are local
parameters for $\chi(M,\PGL(n,\C))$ -- see fact \ref{local-rigidity} for a
precise statement. Let us also mention that the paper \cite{BFG-rigidity}
recovers this theorem for $n=3$. Its approach to the problem, following
\cite{BFG}, leads to actual computations in $\chi(M,\PGL(3,\C))$. In the last
section, we will present the example of the $8$-knot complement.

We prove here that
the peripheral representations are generally discrete: in a neighbourhood of the
geometric representation, the non-discrete peripheral representations are
contained in a real analytic subvariety, see theorem \ref{th:main}. The motto for the proof
is that, $\PGL(n,\C)$ being of rank $n-1$, there is enough room to construct
discrete and faithful representations of the commutative groups $\pi_1(T_i)$ as soon as $n\geq 3$. It
is worth insisting here on the fact that this motto should be
carefully implemented. The examples show that peripheral discreteness does not
hold generically around any unipotent representation. This is why we concentrate
our work on the geometric representation even if the techniques may be used
to deal with other unipotent representations.

It raises an interesting question: in the $\PGL(2,\C)$-case, in the
neighbourhood
of the hyperbolic structure, the peripheral representations are discrete if
\emph{and only if} the whole representation corresponds to a hyperbolic Dehn
filling (or a ramified covering) and is therefore discrete. So there is a local
equivalence between the
peripheral discreteness and the discreteness of the whole
representation\footnote{I do not know a direct proof of it -- that is without
using  Mostow rigidity. I think it would be very interesting to
investigate if we might avoid it.}. Does the same hold in $\PGL(n,\C)$ ? It
would have the surprising consequence that generically, in a neighbourhood of
the geometric representation, the deformed representation remains discrete.

\section{Peripheral representations}

First of all, via the $n$-dimensional irreducible representation $r_n\:
: \: \PGL(2,\C)\hookrightarrow \PGL(n,\C)$, we always consider
$\chi(M,\PGL(2,\C))$ as a subset of $\chi(M,\PGL(n,\C))$. From now on, we
denote this last space by $\chi$, as $n\geq 3$ remains fixed.

We will always assume that our manifold $M$ has only one cusp and therefore drop the index $i$: the peripheral torus is denoted $T$ and $l$ and $m$ are its chosen longitude and meridian.
This simplifies notations without hiding any difficulty.
We will occasionally explain what should be adapted for the case of $\nu$ cusps.

\subsection{Local rigidity in $\chi$}

Let $\rho$ be the representative of an element $[\rho]$ in $\chi(M,\PGL(n,\C))$.
 Menal-Ferrer and
Porti \cite{MenalFerrer-Porti} proved the local rigidity around $[\rho_\geom]$
in $\chi(\pi_1(M),\SL(n,\C))$:
\begin{theorem}[Menal-Ferrer and Porti]\label{thm:mf-p}
 Around $[\rho_\geom]$, the variety $\chi(\pi_1(M),\SL(n,\C))$ is a complex
manifold of dimension $(n-1)$ for which the symmetric functions of the
eigenvalues of the $\rho(l)$ are local parameters.

In particular any $[\rho]$ close enough of $[\rho_\geom]$ is completely
determined by its peripheral representations.
\end{theorem}

\begin{remark}
This result is proven when there are $\nu$ cusps. In this case, the dimension of the manifold becomes $(n-1)\nu$. For parameters, you have to choose a longitude in each peripheral torus and consider the symmetric functions of their eigenvalues.
\end{remark}

In this paper, we prefer to work with the group $\PGL(n,\C)$ and the space
denoted by $\chi$. Let us remark that the previous theorem translate
directly into a theorem about $\chi$ in the neighbourhood of the geometric
representation by the following trick: choose a finite generating set $S$ for $\pi_1(M)$. The geometric representation lifts to $\SL(n,\C)$. Hence for any representation $\rho\in\Hom(\pi_1(M),\PGL(n,\C))$ close enough to $\rho_{\geom}$, there is a unique lift $\rho'(g)\in\SL(n,\C)$ of every $\rho(g)$, $g\in S$, such that $\rho'(g)$ is close to $\rho_\geom(g)$. This defines the representation $\rho'\in\Hom(\pi_1(M),\SL(n,\C))$ which is a lift of $\rho$ and verifies that the
eigenvalues of $\rho'(l)$ and $\rho'(m)$ are close to $1$.

Hereafter, we abuse the notations and still speak about the symmetric functions
of the eigenvalues of $\rho$ whereas we should more precisely speak about those
of $\rho'$.

\subsection{A ramified covering of $\chi$}

Let $T$ be the boundary torus and $l$, $m$ be the fixed longitude and meridian. It will be convenient for our purpose to fix the upper-triangular form of
$\rho(l)$ and $\rho(m)$, which amounts to choosing an order on their
eigenvalues. This will be done classically by passing to a finite ramified
covering of $\chi$ describing the space of representations decorated by flags
fixed by the peripheral representations. 

Note
that $\rho(l)$
and $\rho(m)$ commute. So they are simultaneously trigonalizable over $\C$. At
$\rho=\rho_\geom$, these matrices are unipotent and they have a regular
upper-triangular form: there is a unique complete flag $F_T$ in $\C^n$ they
fix. Stated in a more concrete way, they are conjugated to upper triangular
matrices, with diagonal
entries equal to $1$ and non-zero entries above the diagonal (it is a unique
Jordan bloc of size $n$). After a small
deformation, these matrices generically become simultaneously diagonalizable
with distinct eigenvalues. Hence they fix $n!$ different flags in $\C^n$. Each
of these flags is near $F_T$, or in other terms you have a choice of $n!$ distinct
upper-triangular representatives for $(\rho(l),\rho(m))$. 

Let
$\mathcal {F}$ be the space of complete flags in $\C^n$. Define now the
space $\chi'$ by:

\smallskip

$\left\{(\rho,F)\in\Hom(\pi_1(M),\PGL(n,\C))\times\mathcal{F}\textrm{ such that 
}\hspace{1cm}\right.$\\
$\left.\hspace{6cm}\rho(\pi_1(T)).F=F\right\}/{\PGL(n,\C)}.$\\

\smallskip

\noindent There is a natural projection $\chi'\to\chi$ given by forgetting about the flag.
Moreover, as at $[\rho_\geom]$ the peripheral representation fixes only one flag
(it is regular unipotent) $F_T$, the fiber of $[\rho_\geom]$ is exactly the
point $[\rho_\geom,F_T]$. For the sake of simplicity we will sometimes abuse notation and
still denote $[\rho_\geom]$ this point of $\chi'$.

In some sense, $\chi'$ is exactly the space where one may speak of the
eigenvalues of $\rho(l)$ and not only their symmetric functions: indeed,
consider $[\rho,F]\in\chi'$. Then there are complex numbers $L_k$ and $M_k$ for
$1\leq k\leq n-1$, such that for any basis adapted to the flag $F$, the
matrices $\rho(l)$ and $\rho(m)$ are simultaneously upper-triangular with:
\begin{equation}\label{trig}
 \rho(l)=\begin{pmatrix}
1&*&*&*&*\\&L_1&*&*&*\\&&L_1L_2&*&*\\&&&\ddots&*\\&&&&L_1\ldots
L_{n-1} \end{pmatrix}
\end{equation}
and\footnote{The $L_k$'s are given by the usual choice of simple
positive
roots for the diagonal torus associated to the upper triangular group.}
\begin{equation*}\rho(m)=\begin{pmatrix}
1&*&*&*&*\\&M_1&*&*&*\\&&M_1M_2&*&*\\&&&\ddots&*\\&&&&M_1\ldots
M_{n-1}\end{pmatrix}.
 \end{equation*}
  
This defines an
application $\Hol_\periph$ on the space $\chi'$\footnote{When there are different peripheral tori $(T_i)_{i=1..\nu}$, we define in the
obvious way the functions $L_k^i$, $M_k^i$ and $\Hol_\periph$.}:
\begin{eqnarray*}
\Hol_\periph & : & \chi' \to  (\C^{(n-1)})^2
\\&&[\rho,F]\mapsto((L_k[\rho,F])_{k}, (M_k[\rho,F])_{k})
\end{eqnarray*}

We claim that $\chi'$ is a ramified covering of the
space $\chi$, with covering group the Weyl
group of $\PGL(n,\C)$. The latter is also the permutation
group of $n-1$ points. This ramified covering is often called the space of
decorated representations
(\cite{BFG-rigidity}, for example). This claim is proven in the following fact and is
a consequence of Menal-Ferrer and Porti theorem:
\begin{fact}\label{local-rigidity}
The space $\chi'$ is a ramified covering of $\chi$ around $[\rho_\geom]$.

Moreover the map $\Hol_\periph$ is a local holomorphic embedding into $\C^{2(n-1)}$. Its image is locally a complex manifold of dimension
$(n-1)$. The collection $(L_k[\rho])_{k}$ is a local
parameter for this submanifold in the neighbourhood of
$\Hol_\periph[\rho_\geom]$.
\end{fact}

\begin{remark}
As for theorem \ref{thm:mf-p}, the generalisation to $\nu$ cusps is valid, with the image of $\Hol_\periph$ being locally a manifold of dimension $(n-1)\nu$ in $\C^{2(n-1)\nu}$
\end{remark}
As in the $\PGL(2,\C)$-case \cite{Choi}, we really need $M$ to be hyperbolic
and to work in a neighbourhood of the geometric representation for this
statement
to hold. Note that we give in \cite{BFG-rigidity} an actual description (for
$n=3$) of a neighbourhood of $[\rho_\geom]$ for which it holds (and
counterexamples without the assumptions).

\begin{proof}
 For $1\leq k\leq n-1$, denote by $\sigma_k\: : \: \C^n\to \C$ the $k$-th
symmetric function of $n$ complex numbers and, with a slight abuse of notations,
still denote by $\sigma_k$ the map which sends a matrix  to the $k$-th
symmetric function of its eigenvalues. Then we have two lines of applications:
$$\begin{matrix}
  \chi' &\xrightarrow{\Hol_\periph}& (\C^{n-1})^2 &\to &\C^{n-1}\\
[\rho,F]&\mapsto&((L_k[\rho,F])_k,(M_k[\rho,F])_k)&\mapsto &(L_k[\rho,F])_k,
  \end{matrix}
$$
and
$$\begin{matrix}
  \chi & \to & (\C^{n-1})^2 &\to &\C^{n-1}\\
[\rho] &\mapsto &((\sigma_k(\rho(l)))_k,(\sigma_k(\rho(m)))_k)&\mapsto&
(\sigma_k(\rho(l)))_k.
  \end{matrix}
$$
In a neighbourhood of $1\in\C$, we denote by $z^{\frac{1}{n}}$ the
branch of $n$-th root sending $1$ to $1$. We may then define, in a
neighbourhood of $(1,\ldots,1)$ in $\C^{n-1}$ the map:
$$e:\:\left\{\begin{matrix}\C^{n-1}&\to&\C^n\\
  
(a_1,\ldots,a_{n-1})&\mapsto&\frac{1}{\left(a_1^{n-1}a_2^{n-2}\ldots
a_{n-1}\right)^{\frac 1n}}(1,a_1,a_1a_2,\ldots,a_1a_2\ldots a_{n-1})
  \end{matrix}\right.$$
This map seems complicated mostly because of the choices in the definition of the $L_k$'s and $M_k$'s. We stick with this notation as they will be adapted for the remainder of the paper. In simple terms, this map describes the eigenvalues of $\frac{1}{\det(A)^\frac{1}{n}}A$,
where $A$ is the matrix:
$$ A=\begin{pmatrix}
1&*&*&*&*\\&a_1&*&*&*\\&&a_1a_2&*&*\\&&&\ddots&*\\&&&&a_1\ldots
a_{n-1} \end{pmatrix}.$$
By construction the two above lines fit into the commutative diagram:
$$\begin{matrix}
   \chi' &\xrightarrow{\Hol_\periph}& (\C^{n-1})^2 &\to &\C^{n-1}\\
   \downarrow&&\downarrow&&\downarrow\\
   \chi & \to & (\C^{n-1})^2 &\to &\C^{n-1}\\
  \end{matrix},
$$
where the last two vertical arrows are constructed with the map 
$$\begin{matrix}\C^{n-1}&\to&\C^{n-1}\\
   (a_1,\ldots,a_{n-1})&\mapsto&(\sigma_k(e(a_1,\ldots,a_{n-1})))_{1\leq k\leq
n-1}.
  \end{matrix}$$
  
Now, what should we prove? The theorem \ref{thm:mf-p} of Menal-Ferrer and Porti tells us the
structure of the second line around $[\rho_\geom]$:  the first arrow
biholomorphically sends an open set of $\chi$ onto a
submanifold of $(\C^{n-1})^2$ of dimension $\C^{n-1}$ and for which the second
arrow gives local parameters.
The last vertical arrow is the classical ramified covering with covering group
the permutation group of $n-1$ points. So we just need to prove that the
composition of the two arrows of the first line is injective: indeed, if this
holds, the projection $\chi'\to \chi$ is isomorphic to the classical ramified
covering.

So consider $[\rho]\in\chi$, and $(\sigma_k(\rho(l)))_k$ the vector of symmetric
functions of the eigenvalues of $\rho(l)$. Let $(L_1,\ldots,L_{n-1})$ be a
preimage of this vector by the third vertical arrow. Then we want to prove that
there exists a
\emph{unique} flag $F$, such that $[\rho,F]$ is sent to $(L_1,\ldots,L_{n-1})$.
In terms of the vector $(L_1,\ldots,L_{n-1})$, the eigenvalues of $\rho(l)$ are
$1$, $L_1$, $\ldots$, $L_1\cdots L_{n-1}$ (see eq. \ref{trig}). If those are distinct,
there is no problem:
$F$ has to be the flag whose $k$-dimensional space is the sum of the eigenlines
associated to the $k$ first eigenvalues $1$, $L_1$, $\ldots$, $L_1\cdots L_k$. This is the generic case.

But problems may occur when some eigenvalues coincide: for example if the
eigenspace associated to $1$ has dimension $\geq 2$, we would have different
possible choices for the first line of $F$. We claim this does not happen in a
neighbourhood of $[\rho_\geom]$: even if an eigenvalue has multiplicity\footnote{Let us precise the notations: the multiplicity of an eigenvalue is intended as its multiplicity as a root of the characteristic polynomial. We will say that an eigenvalue is simple if the dimension of its eigenspace is $1$. Beware that an eigenvalue can at the same time be simple and have a multiplicity $\geq 2$.} greater than $2$, its
eigenspace will still be a line:
\begin{lemma}\label{lem:jordan}
 For $\rho$ close enough to $\rho_\geom$, for an eigenvalue
$\lambda$ of multiplicity $r$ of $\rho(l)$, and an integer $1\leq k\leq r$,
there is a unique $k$-plane invariant by $\rho(l)$ inside the characteristic
space associated to $\lambda$.
\end{lemma}
\begin{proof}
This amounts to say that the Jordan decomposition of $\rho(l)$ is given by a
unique block for each eigenvalue. This is equivalent to the fact that every
eigenvalue is simple, i.e. its eigenspace is a line. This is true at $[\rho_\geom]$, as $\rho_\geom(l)$ is regular unipotent, and it is an open condition
given by the property: 
"the
characteristic polynomial and its derivative are coprime polynomials". So it holds in a
neighbourhood of $[\rho_\geom]$.
\end{proof}

With this lemma, one concludes the proof: the $k$-dimensional space of $F$
has to be constructed in the following way: consider the $k$ first eigenvalues
$1,L_1,\ldots,L_1\cdots L_k$. For each eigenvalue $\lambda$ appearing,
denote by $r_\lambda$ the number of times it appears. And define the $k$-plane as the
sum of the unique $r_\lambda$-planes invariant by $\rho(l)$ inside the
characteristic space associated to $\lambda$.
\end{proof}

\subsection{The lift of hyperbolic representations}\label{ss:hyp}

We keep track here of the image $r_n(\chi(\pi_1(M),\PGL(2,\C)))$ in $\chi$.
As said in the introduction of this section, we will simply denote by
$\chi(\pi_1(M),\PGL(2,\C))$ this image.

\begin{fact}\label{fact:hyp}
A point $[\rho]\in \chi$ close enough to $[\rho_\geom]$ belongs to $\chi(M,\PGL(2,\C))$ if and only if there
is a
lift in the ramified covering $\chi'$ such that, for all
$i$, we have $L_1[\rho]=L_2[\rho]=\ldots =L_{n-1}[\rho]$. In this case,
we also have
$M_1(\rho)=M_2[\rho]=\ldots =M_{n-1}[\rho]$.
\end{fact}

\begin{proof}
An easy computation of $$r_n\begin{pmatrix}
t&*\\0&t^{-1}\end{pmatrix}$$
shows that it is an upper triangular matrix with diagonal entries
$$t^{n-1},t^{n-3},\ldots,t^{-(n-3)},t^{-(n-1)}.$$
That is, with the notation of eq. \ref{trig}, $L_1=\ldots=L_{n-1}=t^{-2}$. As locally the eigenvalues of $\rho(l)$ determine a point in $\chi'$, the fact
holds.
\end{proof}

\subsection{Peripheral discreteness}

This paper aims to understand the so-called \emph{peripheral discreteness}: are
the peripheral representations discrete or not? Looking for any
global
result is hopeless: as mentioned in the introduction, the $\PGL(2,\C)$-case is
already understood and shows both non-discreteness (the generic feature in this
dimension) and discreteness (for Dehn surgeries). So we try to understand the
generic behaviour. Precisely, we prove in the $\PGL(n,\C)$-case that
peripheral discreteness becomes the generic behaviour. Let $\mathcal U$ be a
neighbourhood of $[\rho_\geom]$ in $\chi'$ on which $\Hol_\periph$ is
injective and the projection $\chi'\to\chi$ is a ramified covering. Then we have:

\begin{theorem}\label{th:main}
Let $M$ be a complete hyperbolic manifold of dimension $3$ with $1$
cusp. There is a real-analytic subvariety $\mathcal D$ of $\C^{(n-1)}$ of
codimension $\geq 1$ verifying that
for any $[\rho,F]$ in 
$\mathcal U$, with $\Hol_\periph[\rho,F]=(L_k,M_k)_k$, we have:\\
if $(L_k)_k$ lies outside of $\mathcal D$, then the peripheral
representation of $\rho$ is discrete and faithful.
\end{theorem}

\begin{remark}
The generalisation to $\nu$ cusps is natural: $\mathcal D$ becomes a subvariety of $\C^{(n-1)\nu}$ of codimension $\geq 1$.
\end{remark}

It may seems surprising at first glance. But there is an heuristic evidence for
this result. Indeed, the
peripheral representations are representations of $\Z^2$. When in the
$\PGL(2,\C)$-case, outside of the geometric representation, both the elements
$\rho(l)$ and $\rho(m)$ are
loxodromic and preserve the same geodesic. Hence, this $\Z^2$ naturally embeds
in the stabilizer of this geodesic. The latter is a diagonal subgroup isomorphic to
$\C^*$.
At the end, you get some $\Z^2$ included in $\C^*$. It is seldom discrete.
Now, when the ambient group becomes $\PGL(n,\C)$ with $n\geq 3$, then the group $\Z^2$ is
(generically) mapped inside a diagonal subgroup which is isomorphic to
$(\C^*)^{n-1}$. The higher rank indicates that the generic behaviour should be
the discreteness and faithfulness\footnote{Recall that when $n=2$, outside of
$[\rho_\geom]$ but in a neighbourhood, the peripheral representation is \emph{never} discrete and
faithful.}.

Our proof of the theorem will follow this heuristics. We will prove that, under
the hypothesis of the theorem,  no hidden algebraic relationship
between $\rho(l)$ and $\rho(m)$ prevent the discreteness or the
faithfulness. This proof is completed in section \ref{ss:proof}. Beware however that it is not some general triviality. Indeed, if
we do not work in a neighbourhood of the geometric representation, one may exhibit examples of a strong and simple
relationship. A
counterexample is given by the $8$-knot complement: when looking
at the neighbourhood of another representation $\rho$ whose peripheral
representations are unipotent -- actually
lying inside $\mathrm{PU}(2,1)$ -- such a simple relation holds and prevents
faithfulness, see section \ref{sec:example}.

We conclude this subsection by the criterion for discreteness and faithfulness we will use. It is an elementary fact, when you use the homeomorphism $\C^* \simeq \R^*\times \mathbf S^1$:
\begin{fact}\label{crit_discr}
Let $[\rho,F]$ belongs to $\mathcal U$ and note
$\Hol_\periph[\rho,F]=(L_k,M_k)_{k}$. Then a sufficient condition for
the
restriction of $\rho$ to $\pi_1(T)$ to be discrete and faithful is:
 
There exist $1\leq k<h\leq n-1$ such that 
$$\Delta_{k,h}:=\det\begin{pmatrix} \log |L_k| & \log |L_{h}|\\ \log
|M_k|
& \log
|M_{h}|\end{pmatrix} \neq 0.$$
\end{fact}

\begin{proof}
 For any representation $[\rho,F]$ in $\mathcal U$, with
$\Hol_\periph[\rho,F]=(L_k,M_k)_{k}$, the
following are equivalent conditions all implying the discreteness and
faithfulness of the restriction of $\rho$ to $\pi_1(T)$:
\begin{itemize}
\item The vectors $$(L_1,L_1L_2,\ldots, L_1\cdots L_{n-1})$$ and
$$(M_1,M_1M_2,\ldots, M_1\cdots M_{n-1})$$ generate a discrete
subgroup of
$(\C^*)^{n-1}$ isomorphic to $\Z^2$.

\item The vectors $$(|L_1|,|L_1L_2|,\ldots, |L_1\cdots L_{n-1}|)$$ and
$$(|M_1|,|M_1M_2|,\ldots, |M_1\cdots M_{n-1}|)$$ generate a subgroup
of
$(\R^*)^{n-1}$ isomorphic to $\Z^2$.

\item The vectors $$(\log|L_1|,\log|L_1L_2|,\ldots, \log|L_1\cdots
L_{n-1}|)$$ and
$$(\log|M_1|,\log|M_1M_2|,\ldots, \log|M_1\cdots M_{n-1}|)$$ are free
in $\R^{n-1}$.
\end{itemize}
The last point is translated in terms of non vanishing of at least one minor,
that is one of the functions $\Delta_{k,h}$.
\end{proof}

Theorem \ref{th:main} follows easily from this fact, at least if we may find
one single representation $\rho$ for which this determinant does not vanish. The
subvariety $\mathcal D$ is then defined by the vanishing of these determinants. A
nice feature of the proof is that we will exploit what we know about the
$\PGL(2,\C)$-case -- even if in this case the determinants \emph{always} vanish
!

\subsection{Some facts about the $\PGL(2,\C)$-case}\label{ss:thurston}

The main result about the $\PGL(2,\C)$-case goes back to Thurston
\cite{thurston-gt3m} and is explained thoroughly in Neumann-Zagier's paper
\cite{neumann-zagier}. We state it using our notations. Recall from facts
\ref{local-rigidity} and \ref{fact:hyp} that the $\PGL(2,\C)$-case -- seen as a subset of
the ramified covering $\chi'$ -- is characterised by $L_1=\ldots=L_{n-1}$, which implies in turn $M_1=\ldots=M_{n-1}$ for all $i$. Moreover, given
a complex number $L$ close enough to $1$, there is a unique representation
class
$[\rho_L]\in \chi(M,\PGL(2,\C))\subset\chi(M,\PGL(n,\C))$ such that
$L_1[\rho_L]=\ldots=L_{n-1}[\rho_L]=L$. For this
representation, we denote by $M(L)$ the common value of the
$M_k[\rho_L]$. This defines a holomorphic map 
$$L\mapsto M(L).$$
If $L$ is close enough to $1$, then $[\rho_v]$ is close
to
$[\rho_\geom]$ and $M(L)$ is close to $1$. We choose the branch of $\log$
which sends $1$ to $0$. The following theorem states that $\log(M(L))$ is
almost an affine function of $\log(L)$ in the neighbourhood of $1$ 
and its slope is the modulus of the euclidean structure of the torus $T$ in
the hyperbolic structure:

\begin{theorem}[Thurston]\label{th:thurston}
There exists an analytic map $\tau$ defined on a neighbourhood of
$1$ in $\C$ such that we have :
$$\log(M(L))=\tau(L) \log(L).$$
Moreover, $\tau(1)$ belongs to
the upper half-plane $\H$ and is the modulus $\mu$ of the euclidean structure
on $T$ given by the hyperbolic structure on the manifold $M$ (associated to $l,m$).
\end{theorem}

Note that the modulus $\mu$  is \emph{not} a real number. This
will prove useful at the end of the proof. 

In light of fact \ref{crit_discr}, we are rather interested in the dependency of
$\log|M|$ in terms of $\log|L|$. However, for a countable set of
complex numbers $L$, there exist two relatively prime integers $p$ and $q$ such that
$|L|^p|M|^q=1$. Indeed, for any neighbourhood of $1$, for all but a
finite number of relatively prime integers $p$, $q$, one may find an $L$ such that
this relation holds (see \cite[p. 322]{neumann-zagier}). The dependency of these
real parts is really a wild one and it is implied by the tameness of the complex
dependency. 

We will take advantage of that: we will show that the tamed behaviour of
$\log(M)$ with respect to $\log(L)$ generalises to the $\PGL(n,\C)$-case. And this
implies as a counterpart that the real parts dependencies are wild. The
non-vanishing of the determinant of fact \ref{crit_discr} will follow as a
corollary.

\section{Complex analyticity and its real counterpart}

\subsection{Ratios of complex logarithms}

Let $\mathcal V$ be the projection on the space $(L_k)_{k}$
of $\Hol_\periph(\mathcal U)$: those are the possible vectors of eigenvalues of the
longitude for a deformation of $\rho_\geom$. From fact
\ref{local-rigidity}, we know that $\mathcal V$ is a neighbourhood of
$(1,\ldots,1)$ and that to any $v\in \mathcal V$, there exists a unique
point $[\rho_v,F]$ of $\mathcal U$ projecting to $v$.

We define, as in the $\PGL(2,\C)$-case, a map from $\mathcal V$ to
$\C^{(n-1)}$ by $$v=(L_k)_{k}\mapsto (M_k(v):=M_k[\rho_v,F])_{k}.$$
We first generalise the map $\tau$ of the $\PGL(2,\C)$-case. The following
proposition, which reduces to simple linear algebra, is the key point for
theorem \ref{th:main}:
\begin{proposition}\label{tau}
There exist $(n-1)$ applications $\tau_{k}$, holomorphic in a neighbourhood
of $(1,\ldots,1)$ in $\mathcal V$, such that we have, for all
$v=(L_k)_{k}\in \mathcal V$:
$$\log(M_k(v))=\tau_{k}(v) \log(L_k).$$
Moreover, for any $k$,  $\tau_k (1,\ldots,1)$ is the modulus $\mu$ of
the euclidean structure on $T$ induced by the hyperbolic structure on the manifold $M$.
\end{proposition}

Let us note that the first part will only use the fact that we
have local rigidity around $[\rho_\geom]$ and that $\rho_\geom(l)$ is regular
unipotent (it is a unique Jordan bloc of size $n$). Hence it generalises to
more general settings: for example, when $n=3$, I already
mentioned that it is possible to make actual computations in
order to find some representations of fundamental groups of hyperbolic
$3$-manifolds whose peripheral representations are unipotent. In
\cite{BFG-rigidity} we gave a simple criterion of local rigidity and it is an
easy task to check that the image of $l$ is a unique Jordan block.

\begin{proof}
The proof of the existence of the $\tau_{k}$ is similar to the
$\PGL(2,\C)$-case \cite[Lemma 4.1]{neumann-zagier}: we already know
 from fact \ref{local-rigidity} that
the functions $v\mapsto\log(M_k(v))$ are holomorphic and vanish at
$v=(1,\ldots,1)$. In order to get the existence of the $\tau_{k}$'s, we have
to show that for any $k$, the mere condition $L_k=1$ on the $k$-th
entry of $v$ implies that $M_k=1$ for the same entry. Indeed, this will prove
that the ratio
$\log(M_k(v))/\log(L_k)$ is defined when $L_k=1$. This ratio is the
function $\tau_{k}$. 

Choose some $v\in \mathcal V$ such that, for some $k$, $L_k=1$; let
$[\rho_v,F]$ be the representation class associated to $v$. We choose a basis $(e_1,\ldots e_n)$ of $\C^n$ adapted to the flag $F$. In this basis, the matrices $\rho_v(l)$
and
$\rho_v(m)$ present the upper-triangular form we gave in eq. \ref{trig}. We furthermore choose the basis so that $\rho_v(l)$ is in Jordan form. For $0\leq k\leq n$, let $E_k$ be the
$k$-dimensional subspace of $\C^n$ generated by the first $k$ vectors
$e_1,\ldots,e_k$ (with $E_0=\{0\})$. Define $A$, resp
$B$, to be the endomorphism of the plane $E_{k+1}/E_{k-1}$ given by the action of
$\rho(l)$, resp $\rho(m)$. Then we have, in the basis of $E_{k+1}/E_{k-1}$
given by the projections of $e_k$ and $e_{k+1}$, the following matrices
representing $A$ and $B$ (recall that $L_k=1$):
$$A=L_1\cdots L_{k-1}\begin{pmatrix}1 & x \\ 0&1\end{pmatrix}\textrm{ and
}B=M_1\cdots M_{k-1}\begin{pmatrix}1 & y \\ 0&M_k\end{pmatrix}.$$
Recall from lemma \ref{lem:jordan} that, for $[\rho_v]$ close enough to
$[\rho_\geom]$, the eigenvalue $L_1\cdots L_{k-1}$ is simple. This implies that
$x\neq 0$. But we also know that $\rho_v(l)$ and $\rho_v(m)$ commute. So do
$A$ and $B$. This implies $M_k=1$.

The second point is given by the $\PGL(2,\C)$-case: as seen in sections \ref{ss:hyp} and \ref{ss:thurston}, a point in the
subvariety $L_1=\ldots=L_{n-1}$ -- common value hereafter denoted by $L$ -- corresponds to the representation $\rho_L$  in $\chi(M,\PGL(2,\C))$. Hence we have $M_1=\ldots=M_{n-1}=M$. Hence, for any $k$, the ratio $\tau_k(v)=\log(M_k)/\log(L_k)$  equals the function $\tau(L)=\log(M)/\log(L)$ defined in the hyperbolic case. And for this $\tau$, theorem
\ref{th:thurston} gives $\tau(1)=\mu$.
\end{proof}

\subsection{Proof of the generic discreteness}\label{ss:proof}

We are now ready to complete the proof of theorem \ref{th:main}. The idea is
simple: you cannot have both a nearly linear relation with non real
coefficients between the vectors $(\log(L_k))_k$ and
$(\log(M_k))_k$ and a colinearity (over $\R$) between
$(\log|L_k|)_k$ and $(\log|M_k|)_k$. As before let $v$ be the vector $(L_k)_{k}$ and
define the notation $\log(v)=\max_{k}\left|\log(L_k)\right|$.
We have:
\begin{eqnarray*}
 \log(M_k)&=&\mu
\log(L_k)+o(\log(v))
\end{eqnarray*}
It yields the following approximation for $$\Delta_{k,h}=\det\begin{pmatrix}
\log
|L_k| &
\log |L_{h}|\\ \log |M_k| & \log |M_{h}|\end{pmatrix}:$$
$$\Delta_{k,h}=\mathrm{Im}(\mu)\left(\arg(L_k)\log|L_{
h}|-\arg(L_{h})\log|L_k|\right)+o(\log(v)^2).$$

We may look at a deformation defined, for $t>0$, by $L_k(t)=(1+t)e^{it}$ and
$L_{h}(t)=\overline{L_k}$ (all the other being fixed to $1$). As
$\mathrm{Im}(\mu)\neq 0$, we get the non vanishing of the determinant. This
proves
theorem \ref{th:main}.

\subsection{$[\rho_\textrm{geom}]$ is not a smooth point of the real analytic
subvariety}\label{ss:notsmooth}

From the previous computation, it is clear that for any $k$, $h$, the
differential at $v=(1,\ldots,1)$ of 
$$\Delta_{k,h}=\det\begin{pmatrix} \log |L_k| &
\log |L_{h}|\\ \log |M_k| & \log |M_{h}|\end{pmatrix}$$
vanishes. Thus the local geometry of the subvariety $\mathcal D$ of theorem
\ref{th:main}, i.e. the vanishing locus of all determinants, is not clear. And
it is a simple task to prove that $(1,\ldots,1)$ is indeed a singular point on
this subvariety, as the intersection of different branches: consider the
deformation defined by all entries $L_h$ fixed to $1$, except one of them
(say $L_k$). Then all the $M_h$ are constant equal to $1$ except $M_k$.
And all the determinants vanish.

So, in terms of tangent vectors, any vector with only one non-zero entry is
tangent to the subvariety $\mathcal D$ at $(1,\ldots,1)$. As $\mathcal D$ is not of maximal dimension, it surely shows that $(1,\ldots,1)$ is not a smooth point of this subvariety.

\section{Example : the $8$-knot complement}\label{sec:example}

Building upon the work \cite{BFG}, a team works on
a generalisation the famous computer program Snappea in order to understand
representations in $\PGL(3,\C)$ of knot complements. Some results have already
been mentioned in \cite{BFG-rigidity}. A forthcoming paper by Falbel, Koseleff
and Rouillier will explain thoroughly how some parts of $\chi(M,\PGL(3,\C))$ may
be computed, at least for some $M$, the most worked-out example being the
$8$-knot complement.

Recall the well-known presentation of its fundamental group:
$$\langle g_1,g_3 | [g_3,g_1^{-1}]g_3=g_1[g_3,g_1^{-1}]\rangle.$$

For this manifold, one get a complete list of representation whose peripheral
holonomy is unipotent (see \cite{falbel8} and more recently
\cite{deraux-falbel}, which shows that $\rho_2$ and $\rho_3$ are intimately
related). Up to some Galois
conjugations there are only 4 of them:
\begin{itemize}
  \item The holonomy $[\rho_\geom]$ of the hyperbolic structure on $M$.
  \item $[\rho_1]$ defined on the generators by :
    $$\rho_1(g_1)=\begin{pmatrix}
     1&1&-\frac12 -\frac{i\sqrt{3}}{2}\\ 0&1&-1\\0&0&1
    \end{pmatrix}\textrm{ and }
    \rho_1(g_3)=\begin{pmatrix}
     1&0&0\\ 1&1&0\\-\frac12 -\frac{i\sqrt{3}}{2}&-1&1
    \end{pmatrix}$$
  \item $[\rho_2]$ defined on the generators by :
    $$\rho_2(g_1)=\begin{pmatrix}
     1&1&-\frac12 -\frac{i\sqrt{7}}{2}\\ 0&1&-1\\0&0&1
    \end{pmatrix}\textrm{ and }
    \rho_2(g_3)=\begin{pmatrix}
     1&0&0\\ -1&1&0\\-\frac12 +\frac{i\sqrt{7}}{2}&1&1
    \end{pmatrix}$$
  \item $[\rho_3]$ defined on the generators by :
    $$\rho_3(g_1)=\begin{pmatrix}
     1&1&-\frac12\\ 0&1&-1\\0&0&1
    \end{pmatrix}\textrm{ and }
    \rho_3(g_3)=\begin{pmatrix}
     1&0&0\\ \frac54-\frac{i\sqrt{7}}{4}&1&0\\-1&-\frac54-\frac{i\sqrt{7}}{4}&1
    \end{pmatrix}$$
\end{itemize}

All those representations may be checked to be locally rigid and their peripheral representations to be regular unipotent. Moreover, one may
parametrize a neighbourhood of each of these representations
in $\chi(M,\PGL(3,\C))$ ; beware that the actual computation is not an easy
task and will be thoroughly described elsewhere. It is nevertheless possible to
estimate the determinant $\Delta_{1,2}$ -- the only one to compute, as $n=3$ -- on a neighbourhood of these
representations.

This shows the following behaviour:
\begin{itemize}
 \item In the neighbourhood of $[\rho_\geom]$ the result of this article holds! One may get 
 a bit of additional information: the subvariety $\mathcal D$ is locally diffeomorphic to the isotropic cone of a quadratic form on $\C^2$. It was already suggested by the approximation for $\Delta_{1,2}$ used in the proof of the generic discreteness (section \ref{ss:proof}). 
 \item In the neighbourhood of $[\rho_1]$, the discriminants are not identically
$0$, so the peripheral discreteness is still the generic case.
 \item In the neighbourhood of $[\rho_2]$ and $[\rho_3]$, the determinants
always vanish.
\end{itemize}

Note that, for $\rho_2$ and $\rho_3$ the peripheral representation is not
faithful: for example for $\rho_2$ the following relation holds between
suitable chosen longitude $l$ and meridian $m$: 
$$\rho_2(l)=\rho_2(m)^5.$$
The computation shows that, in this case, not only the determinants vanish in a
neighbourhood of $[\rho_2]$, but for each $[\rho]$ close enough, we do still
have:
$$\rho(l)=\rho(m)^5.$$
In other terms, the relation preventing the faithfulness of the peripheral representation at $\rho_2$ is rigid.

It is tempting to think that, for a general manifold $M$ and a discrete, locally rigid, representation
$\rho$ whose peripheral holonomy is regular unipotent, we should have:

the generic behaviour around $[\rho]$ is the peripheral discreteness if and only
if the peripheral representation is faithful. Moreover, in case it is not, the
relation preventing the faithfulness is rigid.

\bibliography{bibli}
\bibliographystyle{amsalpha}

\end{document}